\documentclass{scrartcl}
\usepackage[utf8]{inputenc}
\usepackage[T1]{fontenc}
\usepackage{lmodern}
\usepackage{amsfonts,amssymb,amsthm,amsmath,mathabx,mathtools,mathrsfs}
\usepackage[mathcal]{eucal}
\usepackage{microtype}
\usepackage{xcolor}
\usepackage[margin=1.25in]{geometry}
\usepackage[backend=biber,style=alphabetic,maxnames=10,maxalphanames=3,minalphanames=3]{biblatex}
\definecolor{greencolor}{rgb}{0,0.45,0}
\definecolor{ucphcolor}{rgb}{0.517,0.016,0.016}
\usepackage[colorlinks,citecolor=greencolor,linkcolor=ucphcolor,
            urlcolor=greencolor,unicode,linktocpage]{hyperref}
\usepackage[capitalize]{cleveref}

\usepackage{cjhebrew}

\DeclareFontFamily{U}{rcjhbltx}{}
\DeclareFontShape{U}{rcjhbltx}{m}{n}{<->rcjhbltx}{}
\DeclareSymbolFont{hebrewletters}{U}{rcjhbltx}{m}{n}
\DeclareMathSymbol{\tsadi}{\mathord}{hebrewletters}{118}

\newtheorem{nul}{}[section]
\theoremstyle{definition}
\newtheorem{rmk}[nul]{Remark}
\newtheorem{rec}[nul]{Recollection}
\newtheorem{defn}[nul]{Definition}
\theoremstyle{plain}
\newtheorem{corr}[nul]{Corollary}
\newtheorem{ques}[nul]{Question}
\newtheorem{lem}[nul]{Lemma}
\newtheorem{thm}[nul]{Theorem}
\newtheorem{prop}[nul]{Proposition}

\newcommand{\bb}[1]{\mathbb{#1}}
\newcommand{\mc}[1]{\mathcal{#1}}
\newcommand{\Fun}{\mathrm{Fun}}
\newcommand{\Sp}{\mathrm{Sp}}
\newcommand{\syn}[1]{\mathrm{Syn}_{#1}}

\DeclareMathOperator*{\colim}{colim}
\crefname{thm}{theorem}{theorems}
\Crefname{thm}{Theorem}{Theorems}

\title{The Segal conjecture for large objects and semiadditive lifts}
\author{Vignesh Subramanian}
\date{}

\begin{document}
\maketitle

\begin{abstract}
We prove a projection formula for homotopy fixed points of trivial actions of finite $p$-groups. More precisely,
for every bounded-below, $p$-complete spectrum $X$ whose mod-$p$ homology is $I$-complete, the canonical map
\[
(\bb S_p)^{BG}\otimes X
\xrightarrow{\simeq} X^{BG}
\]
is an equivalence for every finite $p$-group $G$, where $I$ is the augmentation ideal of the Steenrod
algebra.

Building on earlier joint work with Burklund \cite{BSlarge2024}, our proof combines the
Adams--Gunawardena--Miller calculation \cite{AGM} with an adaptation of Carlsson's induction
\cite{carlsson1984equivariant}. As an application, using the work of Yuan \cite{YuanSphere}, we construct
height-one semiadditive lifts of such spectra and show that their cohomology of $p$-finite
spaces depends only on the $1$-type.
\end{abstract}

\tableofcontents

\section{Introduction}

The Segal conjecture is a stable cohomotopy analogue of the Atiyah--Segal completion theorem. Recall that, for
a compact Lie group $G$, the latter identifies
\[
KU^0(BG)\cong R(G)^\wedge_I,
\]
where $R(G)$ is the complex representation ring of $G$ and $I$ is its augmentation ideal. Motivated by this,
Segal conjectured that, for a finite group $G$, the stable cohomotopy of $BG$ is given by
\[
\pi_0(\bb S^{BG})\cong A(G)^\wedge_I,
\]
where $A(G)$ is the Burnside ring of $G$, completed
at its augmentation ideal.

The ccompletion theorem for the sphere spectrum $\bb S$ is considerably more subtle. The Segal
conjecture was established through a sequence of remarkable developments:

\begin{enumerate}
    \item The case $G=C_p$ was proved by Lin andGunawardena \cite{lin1980conjectures, gunawardena1980segal} (see also \cite{lin1980calculation}).

    \item The case of elementary abelian $p$-groups was settled by Adams, Gunawardena, and Miller \cite{AGM}.

    \item Finally, Carlsson \cite{carlsson1984equivariant} completed the proof for all finite groups, using an induction argument building on the elementary abelian case.
\end{enumerate}

The proofs for elementary abelian $p$-groups \cite{lin1980calculation, lin1980conjectures,
gunawardena1980segal, AGM} rely on computations in the Adams spectral sequence, while Carlsson's proof
uses genuine equivariant stable homotopy theory to carry out an induction from elementary abelian
$p$-groups. In fact, these results establish a stronger statement: for every finite $p$-group $G$,
the $p$-completed genuine $G$-sphere is Borel complete\footnote{By the work of \cite{mayMcClure},
for finite $p$-groups the augmentation ideal completion can be replaced by $p$-completion.}.

For $G=C_p$, an equivalent formulation of the
Segal conjecture asserts that the canonical map
\[
\mathrm{can}:\bb S\longrightarrow\bb S^{tC_p}
\]
is a $p$-complete equivalence, where the Tate construction is taken with respect to the trivial
action. By exactness, the same holds for every finite spectrum. Yuan \cite{yuan2021integral}
extended this result to arbitrary direct sums of copies of the sphere spectrum. This naturally
raises the following question:

\begin{ques}
For which spectra $X$ is the canonical map $\mathrm{can}:X\longrightarrow X^{tC_p}$
a $p$-complete equivalence?
\end{ques}

This question was investigated in joint work of the author with Burklund \cite{BSlarge2024}. To state the main result of that work, we recall the following definition.

\begin{defn}[{\cite[Definition~1.3]{BSlarge2024}}]
We say that a spectrum $X$ is $I$-complete if its $\bb F_p$-homology satisfies
\[
H_*(X)\cong \lim_n H_*(X)/I^nH_*(X),
\]
where $I$ is the augmentation ideal of the mod-$p$ Steenrod algebra $\mc A$.
\end{defn}

\begin{thm}[{\cite[Theorem~3.3]{BSlarge2024}}]
Let $X$ be a bounded-below, $p$-complete, $I$-complete spectrum. Then the canonical map
\[
X\xrightarrow{\simeq}X^{tC_p}
\]
is an equivalence.
\end{thm}

The proof in \cite{BSlarge2024} builds on the work of Lin and Gunawardena, using the category of
synthetic spectra introduced by Pstr\k{a}gowski \cite{pstrągowski2022synthetic}, which provides
a natural categorification of the Adams spectral sequence.

\begin{rmk}
The condition of $I$-completeness may initially seem rather restrictive. Roughly speaking, it
allows us to recover the homology from successive quotients whose associated graded pieces have
trivial Steenrod action. In this sense, the homology behaves algebraically like an object
built from simpler pieces. However, an $I$-complete spectrum need not split as a direct
sum of finite spectra.
\end{rmk}

In this article, we extend the Segal conjecture for $I$-complete spectra to all finite $p$-groups.
More precisely, we prove the following theorem.

\begin{thm}\label{thm:main}
(See \Cref{thm:segal-p-groups}.)
Let $X$ be a bounded-below, $p$-complete, $I$-complete spectrum. Then, for every finite
$p$-group $G$, the canonical projection map
\[
(\bb S_p)^{BG}\otimes X
\xrightarrow{\simeq}
X^{BG}
\]
is an equivalence, where the tensor product is taken in the category of $p$-complete spectra.
\end{thm}

\begin{rmk}
The theorem is equivalent to the assertion that the $p$-completed inflation of $X$ is Borel
complete as a genuine $G$-spectrum, for every finite $p$-group $G$. Indeed, this follows
from the classical Segal conjecture and the fact that categorical fixed points preserve colimits.
\end{rmk}

\begin{rmk}
The theorem also admits a formulation in terms of the proper Tate construction. Namely, for
$X$ as in \Cref{thm:main} and every finite $p$-group $G$, the canonical map
\[
\mathrm{can}:X\xrightarrow{\simeq}X^{\tau G}
\]
is an equivalence, where $(-)^{\tau G}$ denotes the proper Tate construction for
the trivial $G$-action. For $G=C_p$, this recovers the ordinary Tate formulation.
\end{rmk}

As an application of \Cref{thm:main}, we use the work of Carmeli--Schlank--Yanovski \cite{CSYHeight} and Yuan \cite{YuanSphere} to construct height-one semiadditive lifts of bounded-below, $p$-complete, $I$-complete spectra to the universal semiadditive category
$\tsadi_1$. In particular, for such a spectrum $X$ and every $p$-finite space $A$, we obtain
a canonical equivalence
\[
X^{\tau_{\leq 1}A}\xrightarrow{\simeq}X^A.
\]
Consequently, if $A$ is pointed and simply
connected, then $\operatorname{Map}_*(A,\Omega^\infty X)\simeq *.$ This extends the connection between height-one
semiadditivity and Lee's theorem observed by Yuan to the class of spectra considered here.

\subsection*{Acknowledgements}

The author thanks Robert Burklund for several enlightening discussions, and Shachar Carmeli for encouraging the author to extend the joint work with Burklund to all finite $p$-groups in order to construct height-one semiadditive lifts.

During the course of this work the author was supported by the Danish National Research Foundation through the Copenhagen Center for Geometry and Topology (DNRF151), European Research Council (ERC) under Horizon Europe (GeoCats, grant No. 101042990), NWO VIDI Grant (no. 223.093) and would like to thank the Max Planck Institute for Mathematics (MPIM) in Bonn for their hospitality and excellent working conditions.

\subsection*{Organisation of the paper}

In Section~\ref{sec:preliminaries}, we recall the necessary preliminaries on genuine equivariant spectra and synthetic spectra.
In Section~\ref{sec:segal}, we prove the Segal conjecture for $I$-complete spectra, first for elementary abelian $p$-groups
using the Adams--Gunawardena--Miller calculation, and then for all finite $p$-groups by adapting Carlsson's induction.
Finally, in Section~\ref{sec:semiadditive}, we use Yuan's sphere to construct height-one semiadditive lifts and deduce
applications to the cohomology of $p$-finite spaces.

\section{Preliminaries} \label{sec:preliminaries}

In this section, we recall the relevant facts about genuine
equivariant spectra and synthetic spectra used in the proofs
of our main results.

\subsection{Genuine equivariant spectra}

Let $G$ be a finite group and let $\Sp^G$ denote the stable $\infty$-category of genuine $G$-spectra.
We recall the various fixed-point functors and completion constructions used throughout the article. For a modern
treatment, we refer the reader to \cite[Section 6]{MNN17} and \cite{QS25}.

\begin{enumerate}
    \item $i:\Sp\to\Sp^G$, \textbf{the inflation}, given by $i(X)=\bb S_G\otimes X$. 
    \item For every subgroup $H \leq G$, we have a colimit preserving functor $(-)^H : \Sp^G \rightarrow \Sp$, given by $\mathrm{Map}_{\Sp_G}(\Sigma^\infty_G G/H_+, X)$, refered to as \textbf{categorical fixed points}.  
    \item Since the spectra $\Sigma^\infty_G G/H_+$, for $H\leq G$. form a set of compact generators of $\Sp^G$, the categorical fixed-point functors preserve colimits and are jointly conservative.
    \item The forgetful functor $u:\Sp^G\longrightarrow \Fun(BG,\Sp),$ which remembers the underlying spectrum with its $G$-action. For every subgroup $H\leq G$, \textbf{the homotopy fixed points} of $Y\in\Sp^G$ are defined by $Y^{hH}:=\lim_{BH}(uY).$
    \item The Borel completion functor $\beta:\Fun(BG,\Sp)\longrightarrow\Sp^G,$ which is the fully faithful right adjoint of $u$. For every $M\in\Fun(BG,\Sp)$ and $H\leq G$, we have $(\beta M)^H\simeq M^{hH}.$
    \item A genuine $G$-spectrum $Y$ is called \textbf{Borel complete} if the unit $\eta:Y\longrightarrow \beta uY$ is an equivalence.  
    \item The \textbf{geometric fixed-point functor} $\Phi^G:\Sp^G\longrightarrow\Sp$ is the colimit-preserving symmetric monoidal functor characterized by the property that $\Phi^G(\Sigma^\infty_G G/H_+)\simeq 0$ for every proper subgroup $H<G$.
    \item Let $\mc F$ be a family of subgroups of $G$, i.e., a collection closed under conjugation and passage to subgroups. Consider the commutative algebra $A_{\mc F}:= \prod_{H\in\mc F}F(G/H_+,\bb S_G).$ The $A_{\mc F}$-acyclization and $A_{\mc F}$-inversion are both smashing and are given by
\[
\Gamma_{\mc F}(Y)\simeq E\mc F_+\otimes Y,
\qquad
L_{\mc F^{-1}}(Y)\simeq
\widetilde E\mc F\otimes Y.
\]
Here $E\mc F$ is the universal $G$-space characterized by
\[
(E\mc F)^H\simeq
\begin{cases}
*, & H\in\mc F,\\
\varnothing, & H\notin\mc F,
\end{cases}
\]
and $\widetilde E\mc F$ is defined by the isotropy-separation cofiber sequence
$E\mc F_+\longrightarrow \bb S_G \longrightarrow \widetilde E\mc F.$
These identifications follow from \cite[Proposition 6.5]{MNN17}.

\item The $A_{\mc F}$-completion functor is given by $(-)^\wedge_{\mc F}:\Sp^G\longrightarrow\Sp^G,$ $Y^\wedge_{\mc F}\simeq F(E\mc F_+,Y).$
      We say that $Y$ is $\mc F$-complete if the canonical map $Y\longrightarrow Y^\wedge_{\mc F}$ is an equivalence. Equivalently, $F(\widetilde E\mc F,Y)\simeq 0.$ In particular, for the trivial family $\mc F=\{1\}$, we have $E\mc F=EG$, and the corresponding completion agrees with Borel completion:
        \[
        Y^\wedge_{\{1\}}\simeq F(EG_+,Y)\simeq\beta uY.
        \]

\item Let $\mc P$ denote the family of all proper subgroups of $G$. The \textbf{geometric fixed points} can be computed
      as the categorical fixed points of the $\mc P$-localization:
        \[
        \Phi^G(Y)\simeq
        \bigl(\widetilde E\mc P\otimes Y\bigr)^G.
        \]
    In particular, for $M\in\Fun(BG,\Sp)$, the \textbf{proper Tate construction} is defined by
        \[
        M^{\tau G}:=\Phi^G(\beta M)
        \simeq
        \bigl(\widetilde E\mc P\otimes\beta M\bigr)^G.
        \]
    When $G=C_p$, this agrees with the ordinary Tate construction $M^{tC_p}$.       

    \item Let $\bar\rho=\mathbb C[G]-\mathbb C$ denote the reduced complex regular representation of $G$. Then
        \[
        E\mc P\simeq S(\infty\bar\rho)
        :=\colim_{n}S(n\bar\rho),
        \]
        where $S(V)$ denotes the unit sphere of a representation $V$. Consequently,
                \[
                \widetilde E\mc P\simeq S^{\infty\bar\rho}
                :=\colim_{n}S^{n\bar\rho},
                \]
        where $S^V$ denotes the representation sphere associated to $V$.

     \item  For every subgroup $H\leq G$, write $W_H=N_G(H)/H$ for its Weyl group. We denote by $\Phi^H:\Sp^G\longrightarrow\Fun(BW_H,\Sp)$ the geometric $H$-fixed-point functor with the residual Weyl-group action, and by $\iota_H$ its fully faithful right adjoint.   

     \item For $K$ subconjugate to $L$, following \cite{QS25}, the \textbf{generalized Tate functor} is defined by
            \[
            \tau_K^L:=\Phi^L\iota_K:
            \Fun(BW_K,\Sp)\longrightarrow\Fun(BW_L,\Sp).
            \]
\end{enumerate}

\begin{rec}[Quigley--Shah reconstruction]\label{rec:QSreconstruction}
Let $\mc F$ be a family of subgroups of $G$, and let
$\mathfrak S_G$ denote the preorder of subgroups of $G$
under subconjugacy. For $H\notin\mc F$, let
$\mc J_H\subseteq\operatorname{sd}(\mathfrak S_G)$
be the full subcategory spanned by chains
\[
K_0<K_1<\cdots<K_n<H,
\qquad K_i\in\mc F.
\]

By \cite[Corollary 2.49]{QS25}, for every
$\mc F$-complete genuine $G$-spectrum $Y$,
there is a natural equivalence
$\Phi^H(Y)\simeq \lim_{\mc J_H} \tau_{K_n}^{H} \tau_{K_{n-1}}^{K_n} \cdots \tau_{K_0}^{K_1} \bigl(\Phi^{K_0}(Y)\bigr) $
in $\Fun(BW_H,\Sp)$.

The transition maps in this diagram are the canonical gluing transformations 
$$\tau_K^M\longrightarrow\tau_L^M\tau_K^L, \quad K<L<M,$$ 
and the units of the adjunctions $\Phi^K\dashv\iota_K$, as described in \cite[Observations 1.25--1.26]{QS25}.
\end{rec}

\subsection{Synthetic spectra}

We recall the relevant properties of $\mathbb F_p$-synthetic spectra introduced by Pstr\k{a}gowski \cite{pstrągowski2022synthetic}.
For the properties used in this article, we refer to \cite[Section 9]{burklund2022boundaries} and \cite[Section 2]{BSlarge2024}.

Let $\syn{\mathbb F_p}$ denote the category of
$\mathbb F_p$-synthetic spectra, equipped with the
synthetic analogue functor
\[ \nu:\Sp \longrightarrow \syn{\mathbb F_p}. \]

\begin{enumerate}
    \item For integers $s,w$, the bigraded synthetic spheres are defined by $\bb S^{s,w}:=\Sigma^{s-w}\nu(\bb S^w).$ There is a canonical map
    $\tau:\bb S^{0,-1}\longrightarrow\bb S^{0,0},$
    whose cofiber we denote by $C\tau$.

    \item Inverting $\tau$ recovers the category of spectra:
    \[
        \syn{\mathbb F_p}[\tau^{-1}]\simeq\Sp,
        \quad
        (\nu X)[\tau^{-1}]\simeq X.
    \]

    \item The object $C\tau$ admits a commutative algebra structure, and there is a fully faithful symmetric
    monoidal functor
    \[
        \mathrm{Mod}_{C\tau}(\syn{\mathbb F_p})
        \longrightarrow
        \mathrm{Stable}(\mathcal A_*),
    \]
    where $\mathcal A_*$ is the dual Steenrod algebra and $\mathrm{Stable}(\mathcal A_*)$ denotes Hovey's
    stable category of comodules. Under this functor, $C\tau\otimes\nu X$ corresponds to $H_*(X;\mathbb F_p)$.
    
    \item A synthetic spectrum $Z$ is called $\tau$-complete if the canonical map
        \[
        Z\longrightarrow \lim_n (Z\otimes C\tau^n)
        \]
    is an equivalence, where $C\tau^n$ denotes the cofiber of multiplication by $\tau^n$. The subcategory of $\tau$-complete synthetic spectra
    is closed under limits, and tensoring with $C\tau$ is conservative on this subcategory.

    If $X$ is a bounded-below $p$-complete spectrum, then $\nu X$ is $\tau$-complete \cite[Remark 2.6]{BSlarge2024}. Moreover, for every bounded-below spectrum $X$, there is a natural equivalence
    \[
    (\nu X)^\wedge_\tau\simeq\nu(X^\wedge_p).
    \]

\item For any diagram of spectra $F:J\rightarrow\Sp$, the canonical coassembly map
        \[
        \left(\lim_{j\in J}\nu F(j)\right)[\tau^{-1}]
        \longrightarrow
        \lim_{j\in J}F(j)
        \]
    is an equivalence \cite[Lemma 2.7]{BSlarge2024}.

\item The synthetic analogue functor $\nu$ is symmetric monoidal \cite[Remark 9.5]{burklund2022boundaries}.

\end{enumerate}

We conclude this subsection with a continuity result for limits in the stable category of Steenrod comodules,
which will be used in the algebraic calculation of Section~\ref{sec:ea-comparison}.

\begin{lem}\label{lem:ea-limits}
Let $(M_n)$ be a tower of graded $\mc A_*$-comodules with degreewise surjective 
transition maps. Suppose there exists an integer $b$ such that 
$(M_n)_d=0$ for all $d<b$ and all $n$. Let $M=\varprojlim_n M_n$ be the degreewise limit. Then:

\begin{enumerate}
    \item $M$ has its natural comodule structure and represents the limit of this tower in $\mathrm{Stable}(\mc A_*)$.
    \item If $N$ is a bounded-below finite-type comodule then the canonical map 
            $$M\otimes N\longrightarrow\lim_n(M_n\otimes N)$$
          is an equivalence in $\mathrm{Stable}(\mc A_*)$.
\end{enumerate}

\end{lem}

\begin{proof}

Let $k = \bb F_p$. $U:\mathrm{Stable}(\mc A_*) \to D(k)^{\mathrm{gr}}$ be the
forgetful functor, and write $M=\varprojlim_n^{\mathrm{gr}} M_n $ for degreewise limits in graded vector space.

First, $M$ has a canonical comodule structure. Indeed, since $\mc A_*$ is nonnegatively graded and
degreewise finite-dimensional, and all $M_n$ vanish below the same internal degree $b$, the canonical map
\[
M\otimes_k\mc A_* \longrightarrow
\varprojlim_n^{\mathrm{gr}}(M_n\otimes_k\mc A_*)
\]
is an isomorphism: in internal degree $d$, only the finitely many indices $b\le i\le d$ contribute, and
each $(\mc A_*)_{d-i}$ is finite-dimensional. Taking limits of the coactions therefore defines a
coaction on $M$. The same argument with $\mc A_* \otimes_k \mc A_*$ verifies coassociativity;
the counit identity follows likewise. The projections $M\to M_n$ are comodule maps and
induce a canonical comparison $\eta:M\longrightarrow L:=\lim_n M_n.$

By \cite[Corollary~3.8]{PstragowskiSynthetic}, the natural embedding $D(\mathrm{Comod}_{\mc A_*})\longrightarrow\mathrm{Stable}(\mc A_*)$ 
identifies the coconnective subcategories and their hearts. On these subcategories, $U$ agrees with the
forgetful functor, which is conservative. Both $M$ and $L$ are coconnective: $M$ belongs to
the heart, and $L$ is a limit of heart objects.

By \cite[Corollary~3.5]{BSlarge2024},
the uniform internal lower bound gives $U(L)\simeq
\varprojlim_n^{\mathrm{gr}}U(M_n).$
Since the transition maps are degreewise surjective, the derived inverse limit of the underlying graded
vector spaces agrees with their degreewise inverse limit $U(M)$. Thus conservativity on coconnective objects gives
$M\xrightarrow{\simeq}L.$ If $N$ is a bounded-below finite-type heart comodule, \cite[Lemma~3.4]{BSlarge2024} gives
$M\otimes N\simeq\left(\lim_n M_n\right)\otimes N \xrightarrow{\simeq} \lim_n(M_n\otimes N).$

\end{proof}

\section{Segal conjecture for large objects} \label{sec:segal}
\subsection{Elementary abelian groups} \label{sec:ea-comparison}

Let $E$ be an elementary abelian $p$-group, and write $\bb S_p:=\bb S^\wedge_p$ for the $p$-complete sphere.
Our goal in this section is to prove that, for every bounded-below $p$-complete spectrum $X$ whose
mod-$p$ homology is $I$-complete, the canonical projection map
\[
(\bb S_p)^{BE}\otimes X
\longrightarrow X^{BE}
\]
is an equivalence. Here the tensor product is taken in the category of $p$-complete spectra.

\begin{rmk}
    For $E=C_p$, this follows from \cite[Theorem 3.3]{BSlarge2024} and the isotropy-separation square.
\end{rmk}

Let $E$ be an elementary abelian $p$-group, and write $H_E:=H_*(BE;\mathbb F_p)$.
We regard $H_E$ as an object of $\mathrm{Stable}(\mc A_*)$. Choose a cellular filtration of $BE$ with finite skeleta
and set
\[
X_j=\Sigma^\infty_+BE^{(j)}.
\]
Then $X_j$ is a finite spectrum, and we denote its
Spanier--Whitehead dual by $DX_j$.

We consider the internal dual of $H_E$,
\begin{equation}
H_E^\vee :=\underline{\mathrm{Hom}}_{\mathrm{Stable}(\mc A_*)} (H_E,\bb F_p)
\simeq \lim_j H_*(DX_j)
\simeq \lim_j DH_*(X_j).
\label{eq:ea-borel-object}
\end{equation}

The following proposition reformulates the algebraic mapping calculation of Adams--Gunawardena--Miller \cite{AGM}
in the stable category $\mathrm{Stable}(\mc A_*)$. For every subgroup $K\leq E$, write $q_K:=\dim_{\bb F_p}(E/K).$

\begin{prop}[{\cite[Theorem~1.6]{AGM}}]\label{prop1.6AGM}
The evaluation maps constructed in \cite[\S~9--10]{AGM} induce an equivalence
\begin{equation}
    \bigoplus_{K\leq E}\Sigma^{0,q_K}H_{E/K} \xrightarrow{\simeq} H_E^\vee
\end{equation}
in $\mathrm{Stable}(\mc A_*)$.
\end{prop}

\begin{rmk}
For $U=E$ and $V=1$, the irreducible $(E,1)$-sets appearing in \cite[Proposition~9.1]{AGM} are precisely
the transitive $E$-sets $E/K$, whose automorphism groups are naturally isomorphic to $E/K$.
The corresponding evaluation maps are constructed in \cite[\S\S~9--10]{AGM}, and \cite[Theorem~1.6]{AGM} identifies their effect on Ext groups. 
\end{rmk}

\begin{corr}\label{corr:ea-finite}
The object $H_E^\vee$ is a finite direct sum of bigraded shifts of bounded-below finite-type comodules.
Consequently, tensoring with $H_E^\vee$ commutes with limits of uniformly internally bounded-below diagrams
of comodules in $\mathrm{Stable}(\mc A_*)^\heartsuit$.
\end{corr}

\begin{proof}
By Proposition~\ref{prop1.6AGM}, we have  $H_E^\vee\simeq \bigoplus_{K\leq E}\Sigma^{0,q_K}H_{E/K}.$
Each $H_{E/K}$ is bounded below and finite-dimensional in every internal degree. The assertion follows from
\cite[Lemma~3.4]{BSlarge2024}.
\end{proof}

\begin{prop}\label{prop:ea-continuity}
Let $M$ be a bounded-below $I$-complete $\mc A_*$-comodule. Then the canonical map
$M\otimes H_E^\vee \longrightarrow \lim_j\left(M\otimes H_*(DX_j)\right)$
is an equivalence in $\mathrm{Stable}(\mc A_*)$.
\end{prop}

\begin{proof}
Put $D_j=H_*(DX_j)$, and for every $\mc A_*$-comodule $N$, consider the canonical map
\[
\eta_N:N\otimes H_E^\vee
\longrightarrow \lim_j(N\otimes D_j).
\]
Both sides are exact in $N$. By Corollary~\ref{corr:ea-finite}, the source commutes
with limits of uniformly internally bounded-below diagrams of heart comodules. Since each $D_j$ is
dualizable, the target commutes with all limits.

As in \cite[proof of Theorem~3.3]{BSlarge2024}, the comparison is an equivalence for bounded-below comodules with trivial Steenrod coaction.
Indeed, every vector space is a retract of a product of copies of $\bb F_p$, and the corresponding splittings are comodule maps when the coaction is trivial. Moreover, sums over distinct internal degrees agree with products in $\mathrm{Stable}(\mc A_*)$
in this bounded-below situation. Since both sides of $\eta$ preserve these products and retracts, the comparison is an equivalence for all such comodules.

Now put $M_n=M/I^nM, K_n=I^nM/I^{n+1}M.$ We have short exact sequences
\[
0\longrightarrow K_n
\longrightarrow M_{n+1}
\longrightarrow M_n
\longrightarrow 0,
\]
which give fiber sequences in $\mathrm{Stable}(\mc A_*)$.
Since $IK_n=0$, each $K_n$ has trivial Steenrod coaction. Thus $\eta_{K_n}$ is an
equivalence, and by exactness and induction, so is $\eta_{M_n}$ for every $n$.

Finally, the tower $(M_n)$ is uniformly internally bounded below and has degreewise
surjective transition maps. By Lemma~\ref{lem:ea-limits} and the $I$-completeness of $M$, we have $M\simeq\lim_n M_n.$ Using Corollary~\ref{corr:ea-finite} and the dualizability of each $D_j$, we obtain
\[
\begin{aligned}
M\otimes H_E^\vee
&\simeq \lim_n(M_n\otimes H_E^\vee) \xrightarrow{\simeq} \lim_n\lim_j(M_n\otimes D_j) \simeq \lim_j(M\otimes D_j).
\end{aligned}
\]
\end{proof}

The preceding algebraic calculation provides the key input for proving the elementary abelian Borel projection formula.
We now pass to synthetic spectra, where we lift the AGM evaluation maps and use $\tau$-completeness to obtain
the desired equivalence of spectra.

For a $p$-complete spectrum $Y$, define
$S_E(Y):=\lim_j\nu(Y\otimes DX_j), S_E:=\lim_j\nu(DX_j).$
Here the tensor products inside $\nu$ are taken in ordinary
spectra. By the symmetric monoidality of $\nu$, there is a
canonical comparison map
\[
\alpha_Y:S_E\otimes\nu Y\longrightarrow S_E(Y).
\]

\begin{rmk}
The objects $S_E$ and $S_E(Y)$ are synthetic lifts of $\bb S^{BE}$ and $Y^{BE}$, respectively. Indeed, 
by \cite[Lemma~2.7]{BSlarge2024}, inverting $\tau$ gives
$S_E[\tau^{-1}]\simeq\bb S^{BE}, S_E(Y)[\tau^{-1}]\simeq Y^{BE}.$
Moreover, reduction of $S_E$ modulo $\tau$ recovers the algebraic object $H_E^\vee$.
\end{rmk}

Suppose that $Y$ is bounded below and $p$-complete, with $I$-complete mod-$p$ homology.
The spectrum $Y\otimes DX_j$ is bounded below and $p$-complete.
Consequently, $\nu(Y\otimes DX_j)$ is $\tau$-complete, since $\tau$-complete objects are closed under limits, $S_E(Y)$ is also $\tau$-complete.

Since $C\tau$ is dualizable, the reduction of $\alpha_Y$ modulo $\tau$ identifies with
\[
H_E^\vee\otimes H_*(Y)
\longrightarrow
\lim_j\bigl(H_*(Y)\otimes H_*(DX_j)\bigr).
\]
By Proposition~\ref{prop:ea-continuity}, this is an equivalence. Thus $\alpha_Y$ is a $C\tau$-equivalence,
and since its target is $\tau$-complete, we obtain
\begin{equation}
\bigl(S_E\otimes\nu Y\bigr)^\wedge_\tau \xrightarrow{\simeq} S_E(Y).
\label{eq:ea-synthetic-comparison}
\end{equation}

\begin{rmk}
The preceding equivalence holds only after $\tau$-completion and therefore does not immediately imply the desired
equivalence of spectra, since $\tau$-completion need not commute with $\tau$-inversion. To overcome this, we construct
synthetic lifts of the AGM evaluation maps whose $\tau$-inversions recover the corresponding topological transfer maps.
\end{rmk}

\begin{lem}[Synthetic lifts of the AGM evaluation maps]
\label{lem:synthetic-agm}
There is a map
\[
d_E:
\bigoplus_{K\le E}
\Sigma^{0,q_K}\nu\bb S[B(E/K)]
\longrightarrow S_E
\]
whose reduction modulo $\tau$ is the equivalence of Proposition~\ref{prop1.6AGM}.
Moreover, its $K$-th component realizes to the adjoint of the transfer evaluation
$\bb S[B(E/K)]\otimes\bb S[BE] \longrightarrow\bb S$
associated to the transitive $E$-set $E/K$.
\end{lem}

\begin{proof}
Consider the transfer associated to the universal $C_p$-bundle, $t:\bb S[BC_p]\longrightarrow\bb S.$

Since $t$ induces the zero map on mod-$p$ homology, the rotated cofiber sequence
\[
\bb S\longrightarrow C_t
\longrightarrow\Sigma\bb S[BC_p]
\]
is short exact on mod-$p$ homology and hence is
preserved by the synthetic analogue functor $\nu$.
Its connecting morphism defines a map
\[
\widetilde t:
\Sigma^{-1}\nu(\Sigma\bb S[BC_p])
=\Sigma^{0,1}\nu\bb S[BC_p]
\longrightarrow\nu\bb S.
\]
Choose the boundary sign so that its $\tau$-inversion is $t$. Modulo $\tau$, this map represents the
extension class of the displayed short exact sequence. By \cite[\S~10]{AGM}, this agrees,
up to sign, with the class $\epsilon\in \operatorname{Ext}_{\mc A}^{1,1} (\bb F_p,H^*(BC_p;\bb F_p))$
used in the AGM calculation.

Now fix $K\le E$ and write $Q=E/K$. Choose an isomorphism $Q\cong(C_p)^{q_K}$.
Taking the $q_K$-fold tensor product of $\widetilde t$ gives a map
$\widetilde t_Q: \Sigma^{0,q_K}\nu\bb S[BQ] \longrightarrow\nu\bb S.$

Consider the homomorphism
\[
f_K:Q\times E\longrightarrow Q,
\qquad
(a,e)\longmapsto a+\bar e,
\]
where $\bar e$ denotes the image of $e$ in $Q$.
By the symmetric monoidality of $\nu$,
precomposing $\widetilde t_Q$ with the map
induced by $f_K$ gives an evaluation map
\[
\widetilde e_K:
\Sigma^{0,q_K}
\nu\bigl(\bb S[BQ]\otimes\bb S[BE]\bigr)
\longrightarrow\nu\bb S.
\]

The associated $(Q\times E)$-set is $Q$, with the action induced by $f_K$. This is precisely
the evaluation set appearing in \cite[\S~9]{AGM}.

By the pullback formula for transfers, the $\tau$-inversion of $\widetilde e_K$ is the
transfer associated to $\ker(f_K)\le Q\times E,$ followed by collapse. Moreover, the identification
of the rank-one class with $\epsilon$, together
with \cite[Proposition~10.3]{AGM}, shows that
the reduction of $\widetilde e_K$ modulo $\tau$
agrees with the corresponding AGM evaluation class.

Restricting the $\bb S[BE]$ variable to the finite skeleton $X_j$, we obtain compatible maps
$\Sigma^{0,q_K}\nu\bb S[BQ]\otimes\nu X_j \longrightarrow\nu\bb S.$
Since $X_j$ is finite, gives compatible maps $\Sigma^{0,q_K}\nu\bb S[BQ] \longrightarrow\nu(DX_j).$ Passing to the limit over $j$ and summing over
$K\le E$ gives the desired map
$$d_E: \bigoplus_{K\le E} \Sigma^{0,q_K}\nu\bb S[B(E/K)] \longrightarrow S_E.$$
Finally, by the preceding identification with AGM's evaluation classes, $C\tau\otimes d_E$
identifies with the equivalence of Proposition~\ref{prop1.6AGM}.
\end{proof}

\begin{thm}[Elementary abelian Borel projection formula]
\label{thm:ea-projection}
Let $E\cong (C_p)^r$, and let $Y$ be a bounded-below $p$-complete spectrum, $I$-complete. Then the canonical projection map
\[
(\bb S_p)^{BE}\otimes Y
\xrightarrow{\simeq} Y^{BE}
\]
is an equivalence.
\end{thm}

\begin{proof}
Tensor the map $d_E$ of Lemma~\ref{lem:synthetic-agm} with $\nu Y$ and compose with the comparison map
$\alpha_Y$. This gives a map
\[
d_{E,Y}:
\bigoplus_{K\le E}
\Sigma^{0,q_K}
\nu\bigl(\bb S[B(E/K)]\otimes Y\bigr)
\longrightarrow S_E(Y).
\]
Here the tensor products inside $\nu$ are taken in ordinary spectra. Its reduction modulo $\tau$ is the AGM equivalence
tensored with $H_*(Y)$, followed by the canonical comparison of Proposition~\ref{prop:ea-continuity}.
Consequently, $d_{E,Y}$ is a $C\tau$-equivalence.

Recall that $S_E(Y)$ is $\tau$-complete. For any bounded-below spectrum $Z$, the canonical
map $\nu Z\to\nu(Z^\wedge_p)$ is a $C\tau$-equivalence, and $\nu(Z^\wedge_p)$ is
$\tau$-complete by \cite[Remark~2.6]{BSlarge2024}. Thus $(\nu Z)^\wedge_\tau\simeq\nu(Z^\wedge_p).$

Applying this to $Z=\bb S[B(E/K)]\otimes Y$, and using that $E$ has finitely many subgroups, we obtain an equivalence
$\bigoplus_{K\le E} \Sigma^{0,q_K} \nu\left( \bigl(\bb S[B(E/K)]\otimes Y\bigr)^\wedge_p \right) \xrightarrow{\simeq} S_E(Y).$ Inverting $\tau$ and applying \cite[Lemma~2.7]{BSlarge2024} gives an equivalence
$$\bigoplus_{K\le E} \bigl(\bb S[B(E/K)]\otimes Y\bigr)^\wedge_p \xrightarrow{\simeq}Y^{BE}.$$

By construction, this equivalence is induced by the transfer evaluations of Lemma~\ref{lem:synthetic-agm}, tensored with $Y$.
In particular, applying the same argument to $Y=\bb S_p$ gives a transfer-induced equivalence
$\bigoplus_{K\le E} \bigl(\bb S[B(E/K)]\otimes\bb S_p\bigr)^\wedge_p \xrightarrow{\simeq}(\bb S_p)^{BE}.$

The equivalence for $Y$ is the composite of this equivalence tensored with $Y$ and the canonical
projection map 
\[
(\bb S_p)^{BE}\otimes Y
\longrightarrow Y^{BE}.
\]
Indeed, this compatibility holds on each transfer-evaluation component. The projection map is therefore an equivalence by two-out-of-three.
\end{proof}

\subsection{Induction to all $p$-groups}\label{sec:finite-groups}

Throughout this section, tensor products of $p$-complete spectra are taken in the symmetric monoidal category
of $p$-complete spectra. Similarly, tensor products, varourious fixed points and generalised Tate constructions of $p$-complete genuine $G$-spectra are taken in the corresponding $p$-complete genuine equivariant category. Whenever ordinary tensor products are required, we specify this explicitly.

We now extend the elementary abelian Borel projection formula to arbitrary finite $p$-groups by induction on their order,
following an argument similar to Carlsson's \cite{carlsson1984equivariant}.

Recall that the classical Segal conjecture asserts that the $p$-completed genuine $G$-sphere is Borel complete.
Since categorical fixed points preserve colimits, for every subgroup $H\leq G$, the map on $H$-fixed points
induced by the Borel-completion unit identifies with the canonical projection map
\[
(\bb S_p)^{BH}\otimes X\longrightarrow X^{BH}.
\]
Consequently, by the joint conservativity of categorical fixed points, it suffices to prove that $iX$ is Borel complete.

We begin with a general vanishing criterion for towers of spectra, which will be used to extend Carlsson's
Euler-nilpotence argument to bounded-below $p$-complete spectra.

\begin{lem} \label{lem:tower_zero}
Let $\{T_k\}_{k\geq 0}$ be a tower of bounded-below
$p$-complete spectra. Suppose there exists $d>0$
such that, for every $k$, the map $T_{k+d}\longrightarrow T_k$
induces zero on mod-$p$ homology. Then, for every
bounded-below $p$-complete spectrum $X$,
\[
\lim_k(T_k\otimes X)\simeq 0.
\]
\end{lem}

\begin{proof}
Call a spectrum $M$ good if
$\lim_k(T_k\otimes M)\simeq 0$.

By hypothesis, the maps
$\bb F_p\otimes T_{k+d}\longrightarrow \bb F_p \otimes T_k$
are null. This give:

\begin{enumerate}
    \item Every $\bb F_p$-module is good.
    \item By exactness, every abelian group $A$ annihilated by a power of $p$ is good.
    \item Since $X$ be bounded below, every homotopy group of $X/p^n$ is annihilated by $p^{2n}$, hence $\tau_{\leq m}X/p^n$
are good.
\end{enumerate}

For each fixed $k$, boundedness below of $T_k$ gives $T_k\otimes X/p^n \simeq \lim_m(T_k\otimes\tau_{\leq m}X/p^n).$ Finally, we obtain
\[
\begin{aligned}
\lim_k(T_k\otimes X) &\simeq \lim_n\lim_m\lim_k (T_k\otimes\tau_{\leq m}X/p^n) \simeq 0.
\end{aligned}
\]
\end{proof}

\begin{corr} \label{corr:free-piece}
Let $G$ be a non-elementary abelian finite $p$-group and put $\rho=\mathbb C[G]-\mathbb C$.
For every bounded-below $p$-complete spectrum $X$, we have $\operatorname{map}_G (S^{\infty\rho},EG_+\otimes iX)\simeq 0.$
\end{corr}

\begin{proof}
Naturally in $k$, we have identification
$\operatorname{map}_G (S^{k\rho},EG_+\otimes iX) \simeq (BG^{-k\rho})^{\wedge}_p\otimes X.$ By \cite[Lemma~III.1]{carlsson1984equivariant}, the $\bb F_p$-Euler class $e(\rho)$ is nilpotent. Under the Thom isomorphism, the transition maps in the tower $\{(BG^{-k\rho})^{\wedge}_p\}_k$
are given on mod-$p$ cohomology by multiplication by powers of $e(\rho)$, hence we are done by \ref{lem:tower_zero}.
\end{proof}

We next establish a comparison result for generalized Tate constructions, which will allow us to
apply the Quigley--Shah reconstruction theorem (see. \ref{rec:QSreconstruction}) in the inductive step.

\begin{lem}\label{lem:tate-coefficient}
Let $G$ be a finite $p$-group and $X$ a bounded-below $p$-complete spectrum. Suppose that $iX$ is Borel complete
for every $p$-group of order strictly less than $|G|$. Let $1\neq K\leq G$, and suppose that $K$ is subconjugate
to $L\leq G$. For every finite $W_K$-set $A$, the canonical map
\[
\tau_K^L(\bb S_p[A])\otimes X
\longrightarrow
\tau_K^L(\bb S_p[A]\otimes X)
\]
is an equivalence. Moreover, $\tau_K^L(\bb S_p[A])$ is a finite $p$-complete permutation $W_L$-spectrum.
\end{lem}

\begin{proof}
Put $N=N_G(K)$ and $W=N/K$, and let $\mc F_K=\{H\leq N \mid K\not\leq H\}$. The right adjoint $\iota_K$ is given by
$\iota_K\simeq \operatorname{Coind}_N^G \left( \widetilde E\mc F_K\otimes \operatorname{Inf}_W^N\beta_W(-) \right). $
Since $|W|<|G|$, the induction hypothesis and the classical Segal conjecture give
\[
\beta_W(\bb S_p[A]\otimes X)
\simeq \bb S_{W,p}[A]\otimes i_WX,
\]
where $\bb S_{W,p}[A]$ denotes the $p$-complete genuine permutation $W$-spectrum. Indeed,
$\bb S_{W,p}[A]$ is dualizable and hence preserves Borel completeness under tensoring.
The remaining functors in the description of $\iota_K$, together with $\Phi^L$, satisfy the
projection formula, proving the first assertion.

For the second, the classical Segal conjecture gives $\beta_W(\bb S_p[A])\simeq\bb S_{W,p}[A]$.
The geometric fixed-point formula for induced spectra then gives $\tau_K^L(\bb S_p[A]) \simeq \bb S_p[U_K^L(A)],$ where
$U_K^L(A)= \left\{ [g,a]\in(G\times_N A)^L: K\leq g^{-1}Lg \right\}$
is a finite $W_L$-set, with the residual action induced by left multiplication.
\end{proof}

\begin{lem}[Singular-piece vanishing] \label{lem:singular-piece}
Let $G$ be a non-elementary abelian finite $p$-group and let $X$ be a bounded-below
$p$-complete spectrum. Suppose that $iX$ is Borel-complete for every $p$-group of order
strictly less than $|G|$. Then
$\operatorname{map}_G (\widetilde E\mathcal P,\widetilde EG\otimes iX) \simeq 0.$
\end{lem}

\begin{proof}
Put $Z_X=\widetilde EG\otimes iX, Z_{\bb S}=\widetilde EG\otimes i\bb S_p.$We proceed in three steps:

\begin{enumerate}
    \item $\mc P$-completion commutes with tensoring by $X$ on the singular piece.
    \item  $Z_{\bb S}$ is $\mc P$-complete using the classical Segal conjecture and Carlsson's Euler-nilpotence theorem
    \item  Finally, we deduce that $Z_X$ is $\mc P$-complete.
\end{enumerate}

\medskip
\noindent
\textbf{Step 1}: We show that the canonical projection map $(Z_{\bb S})^\wedge_{\mc P}\otimes iX \longrightarrow (Z_X)^\wedge_{\mc P}$
is an equivalence. By \cite[Corollary~2.49]{QS25}, the top geometric fixed points of the $\mc P$-completion of a genuine
$G$-spectrum $Y$ are computed by the limit of the diagram
\[
\tau_{K_r}^G\cdots\tau_{K_0}^{K_1}
\Phi^{K_0}(Y)
\]
indexed by chains $K_0<K_1<\cdots<K_r<G$ in the subconjugacy preorder.

For $Y=Z_X$, we have
\[
\Phi^K(Z_X)\simeq
\begin{cases}
0, & K=1,\\
X, & K\neq 1,
\end{cases}
\]
where residual Weyl-group action is trivial. The analogous description
holds for $Z_{\bb S}$, with $X$ replaced by $\bb S_p$. Thus chains beginning at the trivial
subgroup contribute zero. Since these zero-valued chains are closed under passage to larger chains,
they can be omitted from the limit. 

For the remaining chains, by \ref{lem:tate-coefficient} comparisons are equivalences at every stage. These comparisons
are compatible with the Quigley--Shah gluing transformations, thus they assemble into ane quivalance of reconstruction diagrams.

The resulting indexing category has a finite skeleton. Since tensoring with $X$ is exact,
it commutes with this finite limit. Consequently,
$\Phi^G((Z_{\bb S})^\wedge_{\mc P}) \otimes X \xrightarrow{\simeq} \Phi^G((Z_X)^\wedge_{\mc P}).$ For every proper subgroup $H<G$, the comparison
is also an equivalence on $\Phi^H$, since $E\mc P$ restricts to a contractible $H$-space. The claim follows by the joint conservativity
of geometric fixed points.

\medskip
\noindent
\textbf{Step 2:} By the classical Segal conjecture, $i\bb S_p$ is Borel complete, hence $\mc P$-complete. Moreover,
Corollary~\ref{corr:free-piece}, applied to
$X=\bb S_p$, shows that
$EG_+\otimes i\bb S_p$ is $\mc P$-complete.

The isotropy-separation sequence
$EG_+\otimes i\bb S_p \longrightarrow i\bb S_p \longrightarrow Z_{\bb S}$
imply that $Z_{\bb S}$ is $\mc P$-complete.

\medskip
\noindent \textbf{Step 3:} First two steps give, 
$Z_X \simeq Z_{\bb S}\otimes iX \xrightarrow{\simeq} (Z_{\bb S})^\wedge_{\mc P}\otimes iX \xrightarrow{\simeq} (Z_X)^\wedge_{\mc P}.$ By construction, this composite is the $\mc P$-completion unit of $Z_X$. Hence $Z_X$ is $\mc P$-complete, or equivalently,
\[
\operatorname{map}_G
(\widetilde E\mc P,Z_X)\simeq 0.
\]
\end{proof}

\begin{thm}\label{thm:segal-p-groups}
Let $X$ be a bounded-below, $p$-complete spectrum, $I$-complete. Then, for every finite $p$-group $G$, the genuine
$G$-spectrum $iX$ is Borel complete. In particular, the canonical projection map
\[
(\bb S_p)^{BG}\otimes X
\xrightarrow{\simeq} X^{BG}
\]
is an equivalence.
\end{thm}

\begin{proof}
We prove by induction on $|G|$ that $iX$ is Borel complete. If $G$ is elementary abelian, this follows
from Theorem~\ref{thm:ea-projection}, applied to $G$ and all its subgroups. Suppose now that $G$ is not elementary abelian, and assume that $iX$ is Borel complete for every $p$-group of order strictly less than $|G|$. Consider the isotropy-separation sequence
\[
EG_+\otimes iX\longrightarrow iX
\longrightarrow\widetilde EG\otimes iX.
\]
By \ref{corr:free-piece} and \ref{lem:singular-piece} we have  $iX$ is $\mc P$-complete. Now consider the Borel-completion unit
\[
\eta:iX\longrightarrow\beta u(iX).
\]
By the induction hypothesis, its restriction to every proper subgroup is an equivalence.
Consequently, $\eta$ induces an equivalence on $\mc P$-completions. Since both $iX$ and
$\beta u(iX)$ are $\mc P$-complete, $\eta$ itself is an equivalence.

Thus $iX$ is Borel complete for every finite
$p$-group $G$. Taking categorical fixed points
of the Borel-completion unit gives an equivalence
\[
(iX)^G\xrightarrow{\simeq}X^{BG}.
\]
By the classical Segal conjecture and the
colimit-preservation of categorical fixed points,
this identifies with the canonical projection map
\[
(\bb S_p)^{BG}\otimes X
\xrightarrow{\simeq}X^{BG}.
\]
\end{proof}

\section{Height-one semiadditive lifts}\label{sec:semiadditive}

In this section, we use the work of Carmeli-Schlank-Yanovski \cite{CSYHeight} and Yuan \cite{YuanSphere} to construct height-one
semiadditive lifts of the spectra considered in Theorem~\ref{thm:segal-p-groups}. The key input is the Borel projection formula established
in the preceding section.

Write $\mathcal S_1^{(p)}$ for the category of $1$-truncated $p$-finite spaces, whose objects are
finite disjoint unions of classifying spaces of finite $p$-groups. We denote by $\tsadi_1$ the universal
semiadditive category of height one, whose existence and universal property are established in \cite{CSYHeight}.

\subsection{Yuan's sphere and semiadditive lifts}

Recall that higher semiadditivity admits a description in terms of span functors. We use the characterization
of $\tsadi_1$ given in \cite[Proposition~2.12 and Example~2.14]{YuanSphere}.

For a functor  $ M:\operatorname{Span}(\mathcal S_1^{(p)}) \longrightarrow\Sp_{(p)},$ the spans $(A\xleftarrow{a}*\to*)$, coherently
parametrized by $a\in A$, give the \emph{Segal comparison} $\rho_A:M(A)\longrightarrow M(*)^A.$ The span $(*\leftarrow BC_p\to*)$ induces
an endomorphism $|BC_p|_M$ of $M(*)$.

\begin{rec}\label{rec:height-one}
By \cite[Proposition~2.12 and Example~2.14]{YuanSphere},
the category $\tsadi_1$ identifies with the full
subcategory of such functors satisfying:
\begin{enumerate}
    \item $M(*)$ is $p$-complete;
    \item $|BC_p|_M$ is an equivalence;
    \item $\rho_A$ is an equivalence for every
    $A\in\mathcal S_1^{(p)}$.
\end{enumerate}
The first two conditions are respectively the
completeness and divisibility conditions defining
height one. Evaluation at the point gives the
underlying-spectrum functor
\[
U:\tsadi_1\longrightarrow\Sp,
\]
which preserves limits.
\end{rec}

\begin{thm}[Yuan {\cite[Theorem~1.4 and Proposition~3.1]{YuanSphere}}]\label{thm:yuan-sphere}
There exists a lax symmetric monoidal functor
\[
S_Y:\operatorname{Span}(\mathcal S_1^{(p)})
\longrightarrow\Sp_{(p)}
\]
satisfying the following properties:
\begin{enumerate}
    \item $S_Y(*)\simeq\bb S_p$;
    \item $|BC_p|_{S_Y}\simeq\mathrm{id}_{\bb S_p}$;
    \item The Segal comparison $\rho_A:S_Y(A)\longrightarrow(\bb S_p)^A$ is an equivalence for every $A\in\mathcal S_1^{(p)}$.
\end{enumerate}
In particular, $S_Y$ determines a commutative algebra object in $\tsadi_1$ whose underlying spectrum is $\bb S_p$.
\end{thm}

\begin{rmk}
Yuan constructs $S_Y$ using the $K$-theory of finite $A$-sets, setting
\[
S_Y(A)\simeq K(\Fun(A,\mathrm{Fin}))^\wedge_p.
\]
The functoriality with respect to spans is given by restriction and left Kan extension, providing the
coherent span functor structure; see \cite[Constructions~2.18--2.19]{YuanSphere}.
\end{rmk}

We now use Yuan's sphere to construct height-one semiadditive lifts of spectra satisfying the Borel
projection formula.

\begin{thm}[Semiadditive lifting criterion] \label{thm:semiadditive-lift}
Let $X$ be a $p$-complete spectrum such that, for every finite $p$-group $G$, the canonical map
$(\bb S_p)^{BG}\otimes X \xrightarrow{\simeq} X^{BG}$
is an equivalence. Then the span functor
\[
X_Y(A):=S_Y(A)\otimes X
\]
defines an object $\widetilde X\in\tsadi_1$ whose underlying spectrum is $X$.
Moreover, $|BC_p|_{\widetilde X}\simeq\mathrm{id}_X.$ In particular, every bounded-below, $p$-complete, $I$-complete spectrum admits a height-one semiadditive lift.

\end{thm}

\begin{proof}
Postcomposing $S_Y$ with the functor $-\otimes X$ defines a coherent span functor
\[
X_Y:\operatorname{Span}(\mathcal S_1^{(p)})
\longrightarrow\Sp_{(p)}.
\]
Here tensor products are taken in the category of $p$-complete spectra. Since $S_Y(*)\simeq\bb S_p$,
we have $X_Y(*)\simeq X$. Moreover, the identity $|BC_p|_{S_Y}\simeq\mathrm{id}_{\bb S_p}$
implies $|BC_p|_{X_Y}\simeq\mathrm{id}_X.$

It remains to verify the Segal condition. For $A\in\mathcal S_1^{(p)}$, the Segal comparison
factors as
\[
S_Y(A)\otimes X
\xrightarrow{\rho_A\otimes X}
(\bb S_p)^A\otimes X
\longrightarrow X^A.
\]
The first map is an equivalence by Theorem~\ref{thm:yuan-sphere}. For $A=BG$,
the second map is an equivalence by hypothesis. All conditions of \ref{rec:height-one} are satisfied, so $X_Y$ defines an object
$\widetilde X\in\tsadi_1$ with underlying spectrum $X$. The final assertion follows from Theorem~\ref{thm:segal-p-groups}.

\end{proof}

\subsection{Cohomology of \texorpdfstring{$p$}{p}-finite spaces}

We now study the cohomology of $p$-finite spaces with coefficients in spectra admitting a height-one
semiadditive lift.

Recall that every simply connected $p$-finite space $A$ is $\tsadi_1$-acyclic, i.e., the canonical map
\[
M\longrightarrow M^A
\]
is an equivalence for every $M\in\tsadi_1$. This follows from \cite[Proposition~3.2.3(2)]{CSYHeight};
see also \cite[Remark~2.10]{YuanSphere}.

\begin{thm}[{\cite[Remark~2.10 and Corollary~1.5]{YuanSphere}}] \label{thm:one-type}
Let $X$ admit a lift to $\tsadi_1$. Then, for every $p$-finite space $A$, the canonical map
$X^{\tau_{\leq 1}A}\xrightarrow{\simeq}X^A$
is an equivalence. In particular, if $A$ is pointed and simply connected,
then
\[
F(\Sigma^\infty A,X)\simeq 0,
\qquad
\operatorname{Map}_*(A,\Omega^\infty X)\simeq *.
\]
These conclusions apply to every bounded-below
$p$-complete spectrum whose mod-$p$ homology is
$I$-complete.
\end{thm}

\begin{proof}

Choose a lift $\widetilde X\in\tsadi_1$ of $X$, and let $q:A\longrightarrow B:=\tau_{\leq 1}A$
be the Postnikov truncation. For every $b\in B$, the fiber $A_b$ of $q$ is simply connected and
$p$-finite. By the height-one acyclicity theorem, the canonical map $\widetilde X\longrightarrow\widetilde X^{A_b}$
is an equivalence. Hence, we get $\widetilde X^B\xrightarrow{\simeq}\widetilde X^A.$

Finally, the assertion for bounded-below $p$-complete spectra with $I$-complete mod-$p$ homology follows from
Theorem~\ref{thm:semiadditive-lift}.
\end{proof}

\printbibliography
\end{document}